\documentclass[12pt,leqno]{article}
\usepackage{amssymb}
\usepackage{amsmath}
\usepackage{amsthm}
\usepackage[english]{babel}
\usepackage[T1]{fontenc}
\usepackage{a4,indentfirst,latexsym}

\usepackage{hyperref}
\usepackage{anysize}
\usepackage{mathrsfs}
\usepackage{cite,enumitem,graphicx}
\usepackage{verbatim}
\usepackage{color}

\newtheorem{thm}{Theorem}[section]
\newtheorem{prop}[thm]{Proposition}
\newtheorem{lem}[thm]{Lemma}
\newtheorem{cor}[thm]{Corollary}
\newtheorem{rem}[thm]{Remark}

\newtheorem{ass}[thm]{Assumption}

\newenvironment{altproof}[1]
{\noindent
{\em Proof of {#1}}.}
{\nopagebreak\mbox{}\hfill $\Box$\par\addvspace{0.5cm}}

\newcommand\mytop[2]{\genfrac{}{}{0pt}{}{#1}{#2}}

\def\Z{\mathbb{Z}}

\def\R{\mathbb{R}}

\newcommand{\cA}{{\mathcal A}}

\newcommand{\cC}{{\mathcal C}}

\newcommand{\cF}{{\mathcal F}}

\newcommand{\cU}{{\mathcal U}}

\newcommand{\eps}{\varepsilon}

\newcommand{\de}{\delta}
\newcommand{\ka}{\kappa}

\newcommand{\om}{\omega}
\newcommand{\si}{\sigma}
\newcommand{\Ga}{\Gamma}
\newcommand{\De}{\Delta}

\newcommand{\Om}{\Omega}

\newcommand{\ham}{H_\Omega}
\newcommand{\conf}{\cF_N\Omega}
\newcommand{\pa}{\partial}

\newcommand\rot[1]{{\mathrm{Rot}}(W_{#1}(t;w);[0,T(c)])}

\numberwithin{equation}{section}
\DeclareMathOperator{\dist}{dist}

\begin{document}

\title{Periodic solutions with prescribed minimal period of vortex type problems in domains}
\author{Thomas Bartsch \and Matteo Sacchet}

\date{}
\maketitle

\begin{abstract}
We consider Hamiltonian systems with two degrees of freedom of point vortex type
\[
  \kappa_j \dot{z}_j = J \nabla_{z_j} H_\Omega(z_1,z_2), \quad j=1,2,
\]
for $z_1,z_2$ in a domain $\Omega\subset\mathbb{R}^2$. In the classical point vortex context the Hamiltonian $H_\Omega$ is of the form
\[
  H_\Omega(z_1,z_2) = -\frac{\kappa_1 \kappa_2}{\pi} \log |z_1-z_2| - 2\kappa_1 \kappa_2g(z_1,z_2) - \kappa_1^2 h(z_1) - \kappa_2^2 h(z_2),
\]
where $g:\Omega\times\Omega\to\mathbb{R}$ is the regular part of a hydrodynamic Green's function in $\Omega$, $h:\Omega\to\mathbb{R}$ is the Robin function: $h(z)=g(z,z)$, and $\kappa_1$, $\kappa_2$ are the vortex strengths. 
We prove the existence of infinitely many periodic solutions with prescribed minimal period that are superpositions of a slow motion of the center of vorticity close to a star-shaped level line of $h$ and of a fast rotation of the two vortices around their center of vorticity. The proofs are based on a recent higher dimensional version of the Poincar\'e-Birkhoff theorem due to Fonda and Ure\~na.
\end{abstract}

{\bf MSC 2010:} Primary: 37J45; Secondary: 34C25, 37E40, 37N10, 76B47

{\bf Key words:} vortex dynamics; singular first order Hamiltonian systems; periodic solutions; higher dimensional Poincar\'e-Birkhoff theorem

%
%

\section{Introduction}

Given a domain $\Om\subset \R^2$, the dynamics of $N$ point vortices $z_1(t), \ldots, z_N(t) \in \Om$ with vortex strengths $\ka_1,\dots,\ka_N\in\R$ is described by a Hamiltonian system
\begin{equation} \label{eq:general}
\ka_j \dot{z}_j = J \nabla_{z_j} \ham(z_1,\ldots,z_N), \quad j=1,\ldots, N;
\end{equation}
here
$
 J=\begin{pmatrix}
 0&1\\
 -1 & 0
 \end{pmatrix}
$
is the standard symplectic matrix in $\R^2$. The Hamiltonian is of the form
\[
  \ham(z_1,\ldots,z_N) = -\frac{1}{2\pi} \sum_{\mytop{j,k=1}{j\not=k}}^N \ka_j \ka_k \log|z_j-z_k| - F(z_1,\ldots,z_N)
\]
where $F:\Om^N \to \R$ is a function of class $\cC^2$. The Hamiltonian is defined on the configuration space
\[
  \conf = \left\{ (z_1,\ldots,z_N) \in \Om^N: z_j\not=z_k \mbox{ for } j\not=k \right\}.
\]
Observe that the system is singular, but of a very different type compared with the singular second order equations from celestial mechanics.

Systems like \eqref{eq:general} arise as a singular limit problem in Fluid Mechanics. A model for an incompressible, non-viscous fluid in $\Om$ with solid boundary is given by the two dimensional Euler equations
\[
  \left\{
  \begin{aligned}
  &v_t + (v\cdot\nabla) v = -\nabla P,\quad \nabla \cdot v = 0\qquad \text{in }\Om,\\
  &v\cdot\nu = 0 \qquad\text{on }\pa\Om,
  \end{aligned}
  \right.
\]
in which $v(t,x)\in\R^2$ represents the velocity of the fluid and $P(t,x)\in\R$ its pressure; $\nu$ denotes the exterior normal to the domain. Making a point vortex ansatz $\om = \sum_{j=1}^N \ka_j \de_{z_j}$, where $\de_{z_j}$ is the Dirac delta, for the scalar vorticity
$\om = \nabla \times v = \pa_1 v_2 - \pa_2 v_1$, one is led to system \eqref{eq:general}; see \cite{Marchioro-Pulvirenti:1994}.

Classically the point vortex equations \eqref{eq:general} were first derived by Kirchhoff in \cite{Kirchhoff:1876}, who considered the case where $\Om=\R^2$ is the whole plane. In this case the function $F$ in the Hamiltonian is identically zero. On the other hand, when $\Om\ne\R^2$, one has to take account of the boundaries of the domain which leads to
\[
  F(z_1,\dots,z_N) = \sum_{j,k=1}^N \ka_j \ka_k g(z_j,z_k)
\]
where $g:\Om\times\Om \to \R$ is the regular part of a hydrodynamic Green's function in $\Om$. An important role plays the Robin function $h:\Om\to\R$ defined by $h(z)=g(z,z)$. In fact, a single vortex $z(t)\in\Om$ moves along level lines of $h$ according to the Hamiltonian system $\dot{z}=\ka J\nabla h(z)$. This goes back to work of Routh \cite{Routh:1881} and Lin \cite{Lin:1941_1,Lin:1941_2}. The Green function, hence the Hamiltonian $\ham$ is explicitly known only for a few special domains. Moreover $\ham$ is not bounded from above nor from below, and its level sets are not compact, except when $N=2$ and the vortex strengths have different signs. Finally the system \eqref{eq:general} is not integrable in general; see \cite[Section~3.4]{Newton:2001} and \cite{Zannetti-Franzese:1994}. We refer the reader to \cite{Majda-Bertozzi:2001, Marchioro-Pulvirenti:1994, Newton:2001, Saffmann:1995} for modern presentations of the point vortex method.

It is worthwhile to mention that systems like \eqref{eq:general} also arise in other contexts from mathematical physics, e.g.\ in models from superconductivity (Ginzburg-Landau-Schr\"odin\-ger equation), or in equations modeling the dynamics of a magnetic vortex system in a thin ferromagnetic film (Landau-Lifshitz-Gilbert equation); see \cite{Bartsch-Gebhard:2016} for references to the literature. The domain can also be a subset of a two-dimensional surface.

Many authors worked on this problem, mostly in the case $\Om=\R^2$ with $F$=0. In the presence of boundaries much less is known, except in the case of special domains like the half plane or a radially symmetric domain, i.e.\ disk or annulus, when the Green's function is explicitly known. In the case of two vortices and $\ka_1\ka_2<0$ the Hamiltonian is bounded below and satisfies $\ham(z_1,z_2)\to\infty$ as $z=(z_1,z_2)\to\pa\conf$. Consequently all level surfaces of $\ham$ are compact, and standard results about Hamiltonian systems apply. In particular, by a result of Struwe \cite{Struwe:1990} almost every level surface contains periodic solutions. Another simple setting is the case of $\Om$ being radially symmetric and $N=2$ whence the system \eqref{eq:general} is integrable and can be analyzed in detail. For $\Om$ being a disk this has been done in \cite{Dai:2014}.

Except in the above mentioned special cases even the existence of equilibrium solutions of \eqref{eq:general} is difficult to prove; see \cite{Bartsch-Pistoia-Weth:2010,Bartsch-Pistoia:2015}. The problem of finding periodic solutions in a general domain has only recently been addressed in the papers \cite{Bartsch:2016}--\cite{Bartsch-Gebhard:2016} where several one parameter families of periodic solutions of the general $N$-vortex problem \eqref{eq:general} have been found. The solutions found in \cite{Bartsch:2016,Bartsch-Dai:2016,Bartsch-Gebhard:2016} rotate around their center of vorticity, which is situated near a stable critical point of the Robin function $h$. The periods tend to zero as the solutions approach the critical point of $h$. Recall that $h(z)\to\infty$ as $z\to\pa\Om$, hence $h$ always has a minimum. It may have arbitrarily many critical points. For a generic domain all critical points are non-degenerate (see \cite{Micheletti-Pistoia:2014}), hence in this case the results from \cite{Bartsch:2016,Bartsch-Dai:2016,Bartsch-Gebhard:2016} produce many one-parameter families of periodic solutions. Moreover, these solutions lie on global continua that are obtained via an equivariant degree theory for gradient maps. A different type of periodic solutions has been discovered in \cite{Bartsch-Dai-Gebhard:2016}. There the solutions are choreographies where the vortices move near a compact component of the boundary $\pa\Om$ almost following a level line $h^{-1}(c)$ with $c\gg1$.

In the present paper we consider \eqref{eq:general} for $N=2$ vortices in a domain $\Om\subsetneq\R^2$. We find a new type of solutions that are not (necessarily) located near an equilibrium of $h$ but where the two vortices are close to a level line of $h$. More precisely, the solutions that we obtain are essentially superpositions of a slow motion of the center of vorticity along some level line $h^{-1}(c)$ of $h$, and of a fast rotation of the two vortices around their center of vorticity. This will be described in detail. These solutions are of a very different nature from those obtained in \cite{Bartsch:2016}--\cite{Bartsch-Gebhard:2016}. The main geometric assumption is that $h^{-1}(c)$ is strictly star-shaped. Our proofs are based on a recent generalization of the Poincar\'{e}-Birkhoff theorem due to Fonda-Ure\~{n}a \cite{Fonda-Urena:2016}.

The paper is organized as follows. In Section~\ref{sec:res} we state and discuss our results about the existence and shape of periodic solutions of \eqref{eq:general}. In Section~\ref{sec:pf-main} we prove the main Theorem~\ref{thm:main} about the existence of a periodic solution by an application of \cite[Theorem~1.2]{Fonda-Urena:2016}. This requires the computation of certain rotation numbers which will be done in Section~\ref{sec:pf-rot}. The results about the shape of our solutions will be proved in Section~\ref{sec:fastrot}. In the last Section~\ref{sec:pf-rem} we prove various consequences of Theorem~\ref{thm:main} and its proof.

%
%

\section{Statement of results}\label{sec:res}

We consider the Hamiltonian system
\begin{equation}\label{eq:main}
  \ka_j \dot{z}_j = J \nabla_{z_j} \ham(z_1,z_2), \quad j=1,2,
\end{equation}
on a domain $\Om\subset\R^2$ with Hamilton function
\[
  \ham(z_1,z_2) = -\frac{\ka_1 \ka_2}{\pi} \log|z_1-z_2| - 2 \ka_1 \ka_2 g(z_1,z_2) - \ka_1^2 h(z_1) - \ka_2^2 h(z_2)
\]
where $g:\Om\times\Om\to\R$ can be any symmetric $\cC^2$ function, and $h:\Om\to\R$ is defined by $h(z)=g(z,z)$. The parameters $\ka_1,\ka_2\in\R\setminus\{0\}$ have to satisfy $\ka_1+\ka_2\ne0$. We will continue to refer to $z_1$, $z_2$ as point vortices, even though our results are valid in a more general setting.

Let $\cC_c\subset h^{-1}(c)$ be a non-constant periodic trajectory of the one degree of freedom Hamiltonian system
\begin{equation} \label{eq:cent-vort}
  \dot{z} = - (\ka_1 + \ka_2) J\nabla h(z)
\end{equation}
on the level $c\in\R$. Then $\nabla h(z)\ne0$ for every $z\in\cC_c$, hence there exists a neighbourhood $\cU(\cC_c)\subset\Om$ of $\cC_c$ and $c_0<c<d_0$ so that
\[
  \cC_d := \{z\in\cU(\cC_c): h(z)=d\},\quad c_0\le d\le d_0,
\]
is also the trajectory of a non-constant periodic solution of \eqref{eq:cent-vort}. Let $T(d)>0$ be the minimal period of $\cC_d$. Observe that system \eqref{eq:cent-vort} describes the motion of one vortex in $\Om$ with strength $\ka=\ka_1+\ka_2$.

We need one geometric assumption on $h$. A periodic trajectory $\cC$, or any closed $C^1$ curve $\cC\subset\R^2$, is said to be strictly star-shaped if there exists $z_0\in\R^2$ such that for each $w\in S^1\subset\R^2$ the ray $z_0+\R^+w=\{z_0+tw:t\ge0\}$ intersects $\cC$ in precisely one point, and the intersection is transversal.

\begin{ass}\label{ass}
The periodic trajectories $\cC_d = h^{-1}(d)\cap\cU(\cC_c)$, $c_0\le d\le d_0$, of \eqref{eq:cent-vort} are strictly star-shaped. The map $T:[c_0,d_0]\to\R$, $d\mapsto T(d)$, is strictly monotone.
\end{ass}

Clearly, if $\cC_c$ is strictly star-shaped then so is $\cC_d$ for $d$ close to $c$. Observe that we do not require that $\cC_d$ is the boundary of a strictly star-shaped set in $\Om$. Below we shall provide several examples of domains where Assumption \ref{ass} holds with $h$ being the Robin function. In order to state our result recall the action integral for a $T(c)$-periodic function:
\[
  \cA(z) = \frac12 \sum_{j=1}^{2}\int_{0}^{T(c)}\ka_j\langle \dot{z}_j(t), Jz_j(t)\rangle\, dt - \int_{0}^{T(c)} \ham(z(t))\,dt.
\]
The main result of the paper is the following.

\begin{thm}\label{thm:main}
  Suppose $\ka_1,\ka_2,\ka_1+\ka_2\not=0$ and that Assumption~\ref{ass} holds. Then the system \eqref{eq:main} has a sequence of periodic solutions $z^{(n)}(t)$ with minimal period $T(c)$. These satisfy the following properties.
  \begin{itemize}
\item[a)] The center of vorticity $\displaystyle C^{(n)}(t) := \frac{\ka_1}{\ka_1+\ka_2}z^{(n)}_1(t) + \frac{\ka_2}{\ka_1+\ka_2}z^{(n)}_2(t)$ converges uniformly in $t$ as $n\to\infty$ towards a solution $C(t)$ of \eqref{eq:cent-vort} with $C(t)\in\cC_c$.
  \item[b)] $\|z^{(n)}_1-z^{(n)}_2\|_\infty\to0$ as $n\to\infty$, hence $z^{(n)}_1(t),z^{(n)}_2(t)\to C(t)$ uniformly in $t\in [0,T(c)]$.
  \item[c)] Consider the difference
      $\displaystyle D^{(n)}(t) := z^{(n)}_1(t) - z^{(n)}_2(t) = \rho^{(n)}(t)\left(\cos\theta^{(n)}(t),\sin\theta^{(n)}(t)\right)$ in polar coordinates and set $d_n=\big|z^{(n)}_1(0) - z^{(n)}_2(0)\big|$. Then the angular velocity $\dot\theta^{(n)}$ satisfies
      \[
        d_n^2\dot\theta^{(n)}(t) = \frac{\ka_1 \ka_2}{\pi} + o(1) \quad\text{as }n\to\infty \qquad\text{uniformly in $t$.}
      \]
  \item[d)] The action of the solution satisfies $\cA(z^{(n)}) \to -\si\infty$ as $n\to\infty$, where $\si=\mathrm{sgn}(\ka_1\ka_2)$ is the sign of $\ka_1\ka_2$.
\end{itemize}
\end{thm}

\begin{rem} \label{rem:betato0}\rm
  a) This result can be interpreted as follows, using the notation of Theorem \ref{thm:main}. The solutions
  \[
    z^{(n)}_1(t) = C^{(n)}(t)+ \frac{\ka_2}{\ka_1+\ka_2} D^{(n)}(t) \quad\text{and}\quad
    z^{(n)}_2(t) = C^{(n)}(t)- \frac{\ka_1}{\ka_1+\ka_2}D^{(n)}(t)
  \]
  are superpositions of a slow motion of the center of vorticity with minimal period $T(c)$, and of a fast rotation of the two vortices around their center of vorticity. The trajectory of the center of vorticity converges (as $n\to\infty$) towards the level line $\cC_c$ of $h$. The angular velocity of the two vortices around their center of vorticity is asymptotic to $\frac{\ka_1\ka_2}{d_n^2\pi}$ as $d_n\to0$ where $d_n$ is the distance of the initial positions of the two vortices. The rotation number of $z^{(n)}_1(t)-z^{(n)}_2(t)$ in $[0,T(c)]$ is asymptotic to $\frac{|\ka_1\ka_2|T}{2\pi^2d_n^2}$ and tends to infinity as $d_n \to 0$.

  b) In the case $\ka_1\ka_2<0$ the center of vorticity does not lie between the two vortices. If $\ka_1+\ka_2$ is close to $0$ then the two vortices are relatively far away from their center of vorticity, compared with their distance from each other.

  c) Clearly the theorem holds for any $\tilde{c}\in(c_0,d_0)$ instead of $c$.

  d) If $\ka_1\ka_2<0$ then $H_\Om(z)\to\infty$ as $z\to\pa\cF_2(\Om)$, hence the level surfaces $H_\Om^{-1}(c)$ are compact. Therefore a result of Struwe \cite[Theorem~1.1]{Struwe:1990} can be applied and yields that for almost every $c>\inf H_\Om$ there exists a periodic solution of \eqref{eq:main} on $H_\Om^{-1}(c)$. Even in that case Theorem~\ref{thm:main} is new in that we localize the solutions and describe their shape.

  e) It is an interesting problem whether it is possible to weaken or to drop the condition that $\cC_c$ is strictly star-shaped. We refer the reader to \cite{Ding:1983,LeCalvez-Wang:2010,Martins-Urena:2007} for results and discussions of this delicate issue in the setting of the Poincar\'e-Birkhoff fixed point theorem for non-autonomous one degree of freedom Hamiltonian systems. Although star-shapedness is essential for the multidimensional Poincar\'e-Birkhoff fixed point theorem \cite[Theorem~1.2]{Fonda-Urena:2016} we believe that it is not essential in our special case; see also \cite{Fonda-Sabatini-Zanolin:2012}.

  f) It is also an interesting problem to consider more than two vortices. One might conjecture that, given a periodic solution $Z_j(t)=e^{-\om Jt}z_j$, $j=1,\dots,N$, $\om\in\R$, $z_1,\dots,z_N\in\R^2$, of the Hamiltonian system
  \[
    \dot{Z}_j =  -\frac{1}{2\pi} \sum_{\mytop{k=1}{k\not=j}}^N \ka_k \frac{J(Z_j-Z_k)}{|Z_j-Z_k|^2},\qquad j=1,\dots,N,
  \]
  in the plane, there exist solutions $z_j(t)\in\Om$, $j=1,\dots,N$, of the shape
  \[
    z_j(t) = C(t)+rZ_j(t/r^2)+o(r)\quad \text{as $r\to0$},\ j=1,\dots,N.
  \]
  Here $C(t)$ is a periodic solution of the Hamiltonian system $\dot{C} = -\ka J\nabla h(C)$, where $\ka=\sum_{j=1}^{N}\ka_j$ is the total vorticity. Such a result has been proved in \cite{Bartsch:2016,Bartsch-Dai:2016,Bartsch-Gebhard:2016} in the case when $Z(t)\equiv a_0\in\Om$ is an equilibrium, i.e.\ when $a_0\in\Om$ is a critical point of the Robin function $h$. The methods from these papers do not seem to be applicable, however, when $C(t)$ has minimal period $T>0$. Since the minimal period of $Z(t)$ is $2\pi/\om$, the superposition $C(t)+rZ_j(t/r^2)$ is periodic or quasiperiodic depending on whether or not $2\pi r^2/\om T$ is rational.
\end{rem}

It is easy to construct functions $g$ on an arbitrary domain $\Om\subset\R^2$ so that the assumptions of Theorem~\ref{thm:main} hold. We shall now present several examples where these assumptions can be verified for $g$ being the regular part of a hydrodynamic Green function and $h$ being the associated Robin function.

Let us begin with the case of a bounded convex domain $\Om$. It is well known that the Robin function $h:\Om\to\R$ is strictly convex and that it has a unique non-degenerate minimum $z_0$, the harmonic center of $\Om$ (see \cite{Caffarelli-Friedman:1985}). Moreover $h(z)\to\infty$ as $z\to\pa\Om$. We set $m(\Om):=h(z_0)=\min h$. The level sets $h^{-1}(c)$ with $c>m(\Om)$ are connected and strictly star-shaped with respect to $z_0$. For $c>m(\Om)$ we may therefore define $T(c)$ to be the minimal period of the solution of \eqref{eq:cent-vort} with trajectory $h^{-1}(c)$. The following lemma shows that the assumptions of Theorem~\ref{thm:main} are satisfied.

\begin{lem}\label{lem:T(c)}
  For a bounded convex domain $\Om$ the function $(m(\Om),\infty)\to\R$, $c\mapsto T(c)$, defined above is strictly decreasing with $\displaystyle T(m(\Om)):=\lim_{c\to m(\Om)}T(c)=\frac{2\pi}{|\ka_1+\ka_2|\sqrt{\det h''(z_0)}}$ and $T(c)\to0$ as $c\to\infty$. Here $z_0$ is the harmonic center of $\Om$.
\end{lem}

The lemma will be proved in Section~\ref{sec:pf-rem} below. As a consequence of this lemma we can apply Theorem~\ref{thm:main} in an arbitrary bounded convex domain. 

\begin{cor}\label{cor:convex}
  Let $\Om\subset\R^2$ be a bounded convex domain. Then for every $T<T(m(\Om))$ system \eqref{eq:main} has infinitely many periodic solutions $z^{(n)}$ with minimal period $T$ and having the properties stated in Theorem~\ref{thm:main}, where $\cC=h^{-1}(c)$ and $c>m(\Om)$ is uniquely determined by the equation $T(c)=T$. If $T\to T(m(\Om))$ then $c\to m(\Om)=\min h$, and $z^{(n)}$ converges towards the harmonic center of $\Om$.
\end{cor}

Now we get back to a general domain $\Om$. Here we obtain solutions near a non-degenerate local minimum.

\begin{cor}\label{cor:min}
  Let $z_0$ be a non-degenerate local minimum of $h$ and set $m:=h(z_0)$, $T(m):=\frac{2\pi}{|\ka_1+\ka_2|\sqrt{\det h''(z_0)}}$. There exists $\eps>0$ and a neighbourhood $\cU(z_0)$ of $z_0$ such that for $c\in(m,m+\eps)$ system \eqref{eq:main} has infinitely many periodic solutions $z^{(n)}$ with trajectories in $\cU(z_0)$ and minimal period $T(c)<T(m)$. The solutions have the properties stated in Theorem~\ref{thm:main} with $\cC_c=h^{-1}(c)\cap\cU(z_0)$.
\end{cor}

\begin{rem} \rm
  a) Since the Robin function satisfies $h(z)\to\infty$ as $z\to\pa\Om$ in a bounded domain there always exists a minimum. It is not difficult to produce examples of domains so that the associated Robin function has many local minima. Moreover, for a generic domain all critical points are non-degenerate; see \cite{Micheletti-Pistoia:2014}. Therefore Corollary~\ref{cor:min} applies to generic domains.

  b) Corollary~\ref{cor:min} in particular yields solutions $z^{(n)}(t)$ approaching the local minimum $z_0$ of $h$, i.e.\ $z^{(n)}_j(t)\to z_0$ as $n\to\infty$, $k=1,2$. The minimal periods of these solutions converge towards $T_m = \frac{2\pi}{|\ka_1+\ka_2|\sqrt{\det h''(z_0)}}$. In \cite{Bartsch:2016,Bartsch-Dai:2016,Bartsch-Gebhard:2016} the authors also obtained periodic solutions converging towards $z_0$. More precisely, they produced a family of $T_r$-periodic solutions $z^{(r)}(t)$, parameterized over $r\in(0,r_0)$ with $|z^{(r)}_j(t)-z_0| = r+o(r)$ and $T_r\to0$ as $r\to0$. Therefore these solutions are different from those obtained in the present paper. Also the method of proof is very different. In \cite{Bartsch:2016,Bartsch-Dai:2016,Bartsch-Gebhard:2016} variational methods or degree methods were used whereas we apply a multidimensional version of the Poincar\'e-Birkhoff theorem. Consequently, here we do not obtain continua of periodic solutions. Instead we obtain infinitely many periodic solutions with prescribed period.
\end{rem}

In our last corollary we consider the case when $\pa\Om$ has a component that is strictly star-shaped.

\begin{cor}\label{cor:boundary}
  Suppose $\pa\Om$ has a compact component $\Ga$ that is of class $C^2$ and is strictly star-shaped. Then there exist $M>0$ and a neighbourhood $\cU(\Ga)$ of $\Ga$ such that for $c>M$ system \eqref{eq:main} has infinitely many periodic solutions $z^{(n)}$ with trajectories in $\cU(\Ga)$ and minimal period $T(c)$. The solutions have the properties stated in Theorem~\ref{thm:main} with $\cC_c=h^{-1}(c)\cap\cU(\Ga)$. Moreover $T(c)\to0$ as $c\to\infty$.
\end{cor}

\begin{rem}\rm
  a) Corollary~\ref{cor:boundary} applies to the typical multiply connected circular domains $\Om=\Om_0\setminus\bigcup_{i=1}^m\Om_i$ where all $\Om_i$ are strictly starshaped, $\Om_1,\dots,\Om_m\subset\Om_0$ are compactly contained in $\Om_0$, and $\overline\Om_1,\dots,\overline\Om_m$ are disjoint. One can take $\Ga=\pa\Om_i$, for every $i=1,\dots,m$. If $\Om_0$ is bounded one can also take $\Ga=\pa\Om_0$.

  b) In \cite{Bartsch-Dai-Gebhard:2016} the authors also obtain periodic solutions near a compact component $\Ga$ of the boundary. It is not required that $\Om$ is star-shaped, and the authors could deal with $N\ge2$ vortices. On the other hand, in \cite{Bartsch-Dai-Gebhard:2016} the vorticities had to be identical. For $r>0$ small they obtain $T_r$-periodic solutions where the vortices $z_1,\dots,z_N$ all follow the same trajectory $\Ga_r=\{z_1(t):t\in\R\}$ with a time shift $z_j(t)=z_1(t+\frac{(j-1)T_r}{N})$. At first order (in $r$) the trajectory $\Ga_r$ consists of the points $z\in\Om$ with distance $r$ from $\Ga$. These solutions are very different from those obtained in Corollary~\ref{cor:boundary}, however. In particular, for $j\ne k$ the distance $|z_j(t)-z_k(t)|$ is of order $\frac{L}{N}+o(1)$ as $r\to0$ where $L$ is the length of $\Ga$.
\end{rem}

%
%

\section{Proof of Theorem~\ref{thm:main}} \label{sec:pf-main}

For the proof of Theorem~\ref{thm:main} we may assume that the trajectories $\cC_d$, $c_0\le d\le d_0$, are strictly star-shaped with respect to $z_0=0$.
We may also assume that $\ka_1+\ka_2=1$. If $\ka_1+\ka_2\ne1$ then apply Theorem~\ref{thm:main} to the system with $\tilde{\ka}_j=\frac{\ka_j}{\ka_1+\ka_2}$ instead of $\ka_j$, $j=1,2$. A solution $\tilde{z}(t)$ of this system yields a solution $z(t)=\tilde{z}((\ka_1+\ka_2)t)$ of the original system; recall that we assume $\ka_1+\ka_2\ne0$. Finally we set $\si=\mathrm{sgn}(\ka_1\ka_2)$.

Let $E_2$ be the $2\times2$ identity matrix, and set $E_2^\si:=\begin{pmatrix} \si & 0 \\ 0 & 1\end{pmatrix}\in\R^{2\times2}$. The transformation $w=Az$ with
\[
  A = \begin{pmatrix}
         \sqrt{|\ka_1\ka_2|}E_2^\si & -\sqrt{|\ka_1\ka_2|}E_2^\si \\
         \ka_1E_2 & \ka_2E_2
      \end{pmatrix}\in\R^{4\times4}
\]
transforms the system \eqref{eq:main} to a Hamiltonian system
\begin{equation}\label{eq:w}
  \dot{w}_j = J\nabla_{w_j} H_1(w_1,w_2) \quad\text{for } j =1,2,
\end{equation}
with Hamiltonian
\[
  H_1(w_1,w_2) = -\frac{\ka_1\ka_2}{\pi} \log|w_1|
        -  2 \ka_1\ka_2 g\left( A^{-1} w \right)
      - \ka_1^2 h\left( \Pi_1 ( A^{-1} w ) \right)
      - \ka_2^2 h\left( \Pi_2 ( A^{-1} w ) \right).
\]
where $\Pi_j: \R^4 \to \R^2$, $\Pi_j(z_1,z_2)=z_j$, for $j=1,2$. The transformation $A$ is defined on $A\cF_2\Om := A(\cF_2\Om)$. Note that $w_2=\ka_1 z_1+\ka_2 z_2 \in \Om$ provided $|z_1-z_2| < \frac1{|\ka_1|}\dist(z_2,\pa\Om)$, because $\ka_1+\ka_2=1$. Moreover, given a compact subset $K\subset\Om$ there exists $\de>0$ so that $\big(B_\de(0)\setminus\{0\}\big)\times K \subset A\cF_2\Om$. Here $B_\de(0)$ denotes the closed disk around $0$ with radius $\de$.

Given $0<a_1<b_1$ we define the annulus
\[
  \cA_1(a_1,b_1) := \{w_1\in\R^2:a_1 \le |w_1| \le b_1\},
\]
and for $c_0\le c_1<c<d_1\le d_0$ we define the annular region
\[
  \cA_2(c_1,d_1) := \left\{w_2\in\cU: c_1 \le h(w_2) \le d_1\right\}.
\]
From now on we fix some $c_1\in(c_0,c)$ and some $d_1\in(c,d_0)$ arbitrarily. Suitable values  $b_1>a_1>0$ will be carefully chosen later.

\begin{lem} \label{lem:grad-w2}
The gradient of $H_1$ with respect to $w_2$ satisfies
\[
  \nabla_{w_2}H_1(w) = - \nabla h \left(w_2\right) + Q(w),
\]
with $Q(w)=o(1)$ as $w_1\to 0$ uniformly for $w_2$ in compact subsets of $\Om$.
\end{lem}

\begin{proof}
Recall that $\ka_1+\ka_2=1$. A direct computation shows
\[
\begin{aligned}
  \nabla_{w_2}H_1(w)
   &= -2\ka_1\ka_2 \nabla_{z_1}g(A^{-1}w) - 2 \ka_1\ka_2 \nabla_{z_2}g(A^{-1}w) - \ka_1^2\nabla h(\Pi_1 ( A^{-1} w ) )\\
   &\hspace{1cm}
      - \ka_2^2 \nabla h(\Pi_2 ( A^{-1} w ) )\big).
\end{aligned}
\]
The Taylor expansion for $h$ near $w_2$ yields
\[
  \nabla h(\Pi_1 ( A^{-1} w ) ) = \nabla h \left( w_2\right)
+ o(1)
   \quad\text{as } w_1\to0,
\]
and
\[
  \nabla h(\Pi_2 ( A^{-1} w ) ) = \nabla h \left( w_2 \right)
+ o(1)
   \quad\text{as } w_1\to0.
\]
This implies
\[
  \ka_1^2 \nabla h(\Pi_1 ( A^{-1} w ) ) + \ka_2^2 \nabla h(\Pi_2 ( A^{-1} w ) ) = (\ka_1^2+\ka_2^2) \nabla h\left( w_2 \right) + o(1) \quad\text{as } w_1\to0.
\]
Using the symmetry of $g(z_1,z_2)$ and $h(z)=g(z,z)$ we obtain analogously
\[
  \nabla_{z_1}g(A^{-1} w)+ \nabla_{z_2}g(A^{-1}w) = \nabla h \left( w_2 \right) + o(1) \quad\text{as } w_1\to0.
\]
This yields $Q(w)=o(1)$ as $w_1\to 0$. Since all functions are of class $\cC^2$ the convergence is uniform for $w_2$ in a compact subset of $\Om$.
\end{proof}

Now let $W(t;w)\in A\cF_2\Om$ be the solution of the initial value problem for \eqref{eq:w} with initial condition $W(0;w)=w$. We write $J_w$ for its maximal existence interval.

\begin{lem} \label{lem:w1uniformly}
  a) For all $\eps>0$ there exists $0<\de<\eps$ such that $\big(B_\de(0)\setminus \{0\}\big)\times\cA_2(c_1,d_1) \subset A\cF_2\Om$. Moreover, if $0<|w_1|\le\de$ and $w_2\in\cA_2(c_1,d_1)$ then
  \[
    |W_1(t;w)| < \eps \mbox{ for every $t \in J_w$ with $W_2(t;w) \in \cA_2(c_0,d_0)$.}
  \]

  b) If $\sup J_w<\infty$ for some $w\in \big(B_\de(0)\setminus\{0\}\big)\times\cA_2(c_1,d_1) \subset A\cF_2\Om$ then there exists $T(w)<\sup J_w$ such that $W(t;w) \notin \big(B_\de(0)\setminus\{0\}\big)\times\cA_2(c_1,d_1)$ for $T(w)<t<\sup J_w$.
\end{lem}

\begin{proof}
a) By contradiction, suppose that for some $\eps>0$ there exist sequences $w_n=(w_{1,n},w_{2,n})$, $t_n \in J_{w_n}$, with $|w_{1,n}| \to 0$ as $n\to\infty$, $w_{2,n} \in \cA_2(c_1,d_1)$ and
\begin{equation} \label{eq:W1geqeps}
  |W_1(t_n, w_n)| \geq \varepsilon,\quad W_2(t,w_n) \in \cA_2(c_0,d_0).
\end{equation}
Then
$$
  H_1(W_1(t_n,w_n), W_2(t_n,w_n)) = H_1(w_{1,n},w_{2,n}).
$$
because the Hamiltonian is constant along a solution. But in this last equality the left hand side is bounded for all $n$ as a consequence of \eqref{eq:W1geqeps} whereas the right hand side tends to $\si\infty$ as $n\to\infty$.

b) This follows from a similar energy argument.
\end{proof}

For $w_2\in\Om$ let $Z(t;w_2)$ be the solution of the initial value problem
\begin{equation}\label{eq:Z}
  \dot{Z}(t;w_2) = -J\nabla h \left( Z(t;w_2)\right),\quad Z(0;w_2) = w_2.
\end{equation}
If $w_2\in\cA_2(c_0,d_0)$ this is defined for all $t\in\R$. The following lemma concerns the existence of $W(t;w)$ for $t$ in the prescribed time interval $[0,T(c)]$, and the behaviour $W_2(t;w)$ as $w_1 \to 0$.

\begin{lem}\label{lem:existenceT}
  a) There exists $\de>0$ with $\big(B_\de(0)\setminus \{0\}\big)\times \cA_2(c_1,d_1) \subset A\cF_2\Om$ and such that the solution $W(t;w)$ exists for $t\in[0,T(c)]$ provided $0<|w_1|\le\de$ and $w_2\in \cA_2(c_1,d_1)$. Moreover, $W_2(t;w)\in \cA_2(c_0,d_0)$ for all $t\in[0,T(c)]$.

  b) For $w_2\in \cA_2(c_1,d_1)$ there holds $W_2(t;w)\to Z(t;w_2)$ as $w_1 \to 0$ uniformly on $[0,T(c)]$, and uniformly for $w_2\in \cA_2(c_1,d_1)$.
\end{lem}

\begin{proof}
a) Set $\eps:=\frac12\dist\big(\cA_2(c_1,d_1),\cA_2(c_0,d_0)\big)>0$ and let
\[
  \cU_\eps(\cA_2(c_1,d_1)) =\{ w\in\Om: \dist(w,\cA_2(c_1,d_1)) \le \eps\}\subset \cA_2(c_0,d_0)
\]
be the closed $\eps$-neighbourhood of $\cA_2(c_1,d_1)$. We proceed in three steps.

{\sc Step} 1: There exists $\de_0>0$ and $t_0>0$ so that $W(t;w)$ exists for $t\in[0,t_0]$ provided $0<|w_1|\le\de_0$ and $w_2\in\cU_\eps(\cA_2(c_1,d_1))$.

Choose $\de_1>0$ such that $\big(B_{\de_1}(0)\setminus \{0\}\big)\times \cA_2(c_0,d_0) \subset A\cF_2\Om$ and set
\begin{equation}\label{eq:bound-grad}
  C:=\sup_{\mytop{0<|w_1|\leq \de_1}{w_2\in \cA_2(c_0,d_0)}} \left| \nabla_{w_2} H_1(w_1,w_2) \right|.
\end{equation}
Note that $C < \infty$ because $\nabla_{w_2} H_1$ is defined and continuous also for $|w_1|=0$. By Lemma~\ref{lem:w1uniformly}~a) we can find $\de_0>0$ such that if $0<|w_1|\le\de_0$ and $w_2\in\cU_\eps(\cA_2(c_1,d_1))$, $W_2(t;w)\in\cA_2(c_0,d_0)$, then $|W_1(t;w)|<\de_1$. Now Lemma~\ref{lem:w1uniformly}~b) implies that $W(t;w)$ exists for $t\in[0,\eps/C]$. Setting $t_0=\eps/C$ we proved {\sc Step} 1.

{\sc Step} 2: If $w_1^{(n)}\to 0$ and $w_2^{(n)}\in \cU_\eps(\cA_2(c_1,d_1))$ with $w_2^{(n)}\to w_2$, $w_2\in \cU_\eps(\cA_2(c_1,d_1))$, then $W_2(t;w^{(n)}) \to Z(t;w_2)$, uniformly for $t\in[0,t_0]$, and uniformly for $w_2\in \cA_2(c_1,d_1)$.

In fact, using the equation for $w_2$ in integral form we have for $t\in[0,t_0]$:
\[
  \begin{aligned}
    &\big|W_2(t;w^{(n)})-W_2(t;w^{(m)})\big| \\
    &\hspace{1cm}
     \leq \big|w_2^{(n)}-w_2^{(m)}\big|+\int_0^t \big|\nabla_{w_2}H_1 (W(s;w^{(n)}))-\nabla_{w_2}H_1(W(s;w^{(m)}))\big| ds.
  \end{aligned}
\]
Note that
$\left\{W(t;w):t\in[0,t_0],\ w\in\big(B_{\de_1}(0)\setminus\{0\}\big)\times\cU_\eps(\cA_2(c_1,d_1))\right\}\subset A\cF_2\Om$ is a relatively compact subset 
in $\Om\times\Om$. Since $\nabla_{w_2} H_1$ is defined on $\Om\times\Om$ and is Lipschitz continuous on compact sets there exists $k>0$ such that
\[
\begin{aligned}
  &\big|W_2(t;w^{(n)})-W_2(t;w^{(m)})\big|\\
  &\hspace{.5cm}
   \le \big|w_2^{(n)}-w_2^{(m)}\big| + k\int_0^t \big|W_1(s;w^{(n)})-W_1(s;w^{(m)})\big|
         + \big|W_2(s;w^{(n)})-W_2(s;w^{(m)})\big| ds.
\end{aligned}
\]
Now Gronwall's Lemma yields for $t\in[0,t_0]$:
\[
  |W_2(t;w^{(n)})-W_2(t;w^{(m)})|
   \le \left(\big|w_2^{(n)}-w_2^{(m)}\big| + k\int_0^{t_0} \big|W_1(s;w^{(n)})-W_1(s;w^{(m)})\big|\right) e^{kt_0}.
\]
This implies that $W_2(t;w^{(n)})$ converges as $n\to \infty$ uniformly for $t\in[0,t_0]$. The limit $Z(t;w_2)$ satisfies the equation \eqref{eq:Z} because
\[
  \nabla_{w_2}H_1 (W(t;w^{(n)}))\to - \nabla h \left( Z(t;w_2) \right) \quad\text{as $n\to\infty$;}
\]
see Lemma~\ref{lem:grad-w2}. This proves {\sc Step} 2.

{\sc Step} 3: There exists $\de$ with $0<\de \leq \de_0$ such that if $0<|w_1|\leq \de$ and $w_2\in\cA_2(c_1,d_1)$ then $W_2(t;w)\in\cA_2(c_0,d_0)$, for all $t\in [0,T(c)]$.

Arguing by contradiction, suppose there exist $w_1^{(n)}\to0$, $w_2^{(n)}\to w_2\in\cA_2(c_1,d_1)$ and $t_n\geq t_0$ such that $W_2(t_n;w^{(n)})\in\pa\cA_2(c_0,d_0)$. {\sc Step} 2 implies $W_2(t;w^{(n)}) \to Z(t;w_2)$ as $n\to\infty$ uniformly on $[0,t_0]$. Then there exists $n_1$ such that for all $n\geq n_1$ we have $W_2(t_0;w^{(n)}) \in \cU_\eps(\cA_2(c_1,d_1))$. This implies that $t_n \geq 2t_0$ for all $n\geq n_1$. So we can apply again {\sc Step} 2 and obtain that $W_2(t;w^{(n)}) \to Z(t;w_2)$ uniformly on $[0,2t_0]$. By induction the procedure continues until we obtain in a finite number of steps that $W_2(t;w^{(n)}) \to Z(t;w_2)$ uniformly on $[0,T(c)]$, which gives the contradiction and proves {\sc Step} 3.

b) This follows from Gronwall's lemma as in {\sc Step} 2.
\end{proof}

Since $W_1(t;w)\ne0$ for any $t,w$ there exists a continuous choice of the argument of $W_1(t;w)$ and we may define the rotation number
\[
  \rot{1} := \frac{1}{2\pi}\big(\arg(W_1(T(c);w))-\arg(w_1)\big) \in \R.
\]
And since $W_2(t;w)\in\cA_2(c_0,d_0)$ for $0<|w_1|\le\de$, $w_2\in\cA_2(c_1,d_1)$, $t\in [0,T(c)]$ we may also define the rotation number
\[
  \rot{2} := \frac{1}{2\pi}\big(\arg(W_2(T(c);w))-\arg(w_2)\big) \in \R.
\]

In the next section we shall prove the following result; here $\de>0$ is from Lemma~\ref{lem:existenceT}~a).

\begin{prop}\label{prop:rot}
  For every $a_0>0$ there exist $0<a_1<b_1<\min\{a_0,\de\}$ arbitrarily small and there exists $\nu\in\Z$ such that the following holds for $w\in \cA_1(a_1,b_1)\times\cA_2(c_1,d_1)$.
  \begin{itemize}
    \item[a)] If $\si>0$ then
      \[
        \rot{1}\  \begin{cases}
                    > \nu, & \mbox{if } |w_1|=a_1 \\
                    < \nu, & \mbox{if } |w_1|=b_1.
                  \end{cases}
      \]
    The inequalities are reversed if $\si<0$.
    \item[b)] If $T(d)$ is strictly increasing for $d\in(c_0,d_0)$ then
      \[
        \rot{2}\  \begin{cases}
                    > 1, & \mbox{if } w_2\in\cC_{c_1} \\
                    < 1, & \mbox{if } w_2\in\cC_{d_1}.
                  \end{cases}
      \]
    The inequalities are reversed if $T(d)$ is strictly decreasing for $d\in(c_0,d_0)$.
  \end{itemize}
\end{prop}

Using Proposition~\ref{prop:rot} we can now prove Theorem~\ref{thm:main}. For any $w_2\in\cA_2(c_1,d_1)$ the rotation number of $W_1(t;w)$ in the interval $[0,T(c)]$ passes $1$ as $w_1$ goes from the inner boundary of $\cA_1(a_1,b_1)$ to the outer boundary of $\cA_1(a_1,b_1)$. Similarly, for any $w_1\in \cA_1(a_1,b_1)$ the rotation number of $W_2(t;w)$ in the interval $[0,T(c)]$ passes $\nu\in\Z$ as $w_2$ goes from one boundary curve of $\cA_2(c_1,d_1)$ to the other one. This is precisely the setting of the generalized Poincar\'e-Birkhoff Theorem \cite[Theorem~1.2]{Fonda-Urena:2016}. As a consequence we deduce that the Hamiltonian system \eqref{eq:w} has a $T(c)$-periodic solution with initial condition $w\in\cA_1(a_1,b_1)\times\cA_2(c_1,d_1)$. Lemma~\ref{lem:existenceT} implies that $W_2(t;w)\in\cA_2(c_0,d_0)$ for all $t\in\R$, provided $b_1$ is small.

Now recall that $c_1\in(c_0,c)$ and $d_1\in(c,d_0)$ were chosen arbitrarily, whereas $0<a_1<b_1$ could be chosen arbitrarily small. Therefore we can  consider sequences $c_n\nearrow c$, $d_n\searrow c$ and can construct sequences $0<a_n<b_n<a_{n-1}\to0$ such that \eqref{eq:w} has a $T(c)$-periodic solution $w^{(n)}(t)$ with $w^{(n)}(0)\in\cA_1(a_n,b_n)\times\cA_2(c_n,d_n)$ and $w^{(n)}_2(t)\in\cA_2(c_{n-1},d_{n-1})$ for all $t\in\R$. Let $z^{(n)}(t)=A^{-1}w^{(n)}(t)$ be the corresponding solution of \eqref{eq:main}. Parts a) and b) of Theorem~\ref{thm:main} follow immediately. Parts c) and d) will be proved in Section~\ref{sec:fastrot}.

%
%

\section{Proof of Proposition~\ref{prop:rot}}\label{sec:pf-rot}

It will be useful to introduce polar coordinates for $W_1,W_2$. We set $e(\theta) = (\cos\theta,\sin\theta)$
and fix initial conditions $w_1=\rho_1e(\theta_1)$, $w_2=\rho_2e(\theta_2)$. Then setting $\rho=(\rho_1,\rho_2)\in(\R^+)^2$ and $\theta=(\theta_1,\theta_2)\in\R^2$ we define $R_j(t;\rho,\theta)=\big|W_j(t;\rho_1e(\theta_1),\rho_2e(\theta_2))\big|$ and let $\Theta_j(t;\rho,\theta)$  be a continuous choice of the argument of $W_j(t;\rho_1e(\theta_1),\rho_2e(\theta_2))$. Thus we can write
$$
  W_j(t;w) = R_j(t;\rho,\theta) e({\Theta_j}(t;\rho,\theta)) \mbox{ for } j=1,2.
$$
We will also write $R(t;\rho,\theta)=(R_1,R_2)(t;\rho,\theta)$ and $\Theta(t;\rho,\theta)=(\Theta_1,\Theta_2)(t;\rho,\theta)$.

Next we describe the radial component of the boundary curves of $\cA_2(c_1,d_1)$ as a function of the angle, obtaining functions $r_j:\R\to(0,\infty)$ defined by $r_1(\theta)e({\theta}) \in \cC_{c_1}$ and $r_2(\theta)e({\theta}) \in \cC_{d_1}$. Since both boundary curves are strictly star-shaped with respect to the origin, $r_j$ is well defined. Clearly $r_j$ is $2\pi$-periodic and there holds
\[
  \cC_{c_1} = \{r_1(\theta)e({\theta}): \theta\in\R\},\quad \cC_{d_1} = \{r_2(\theta)e({\theta}): \theta\in\R\}.
\]
We also set
\[
  \cA_2^{pol}(c_1,d_1) := \{(\rho_2,\theta_2)\in\R^+\times\R: \rho_2e({\theta_2}) \in \cA_2(c_1,d_1)\}.
\]
Proposition~\ref{prop:rot} is now equivalent to the following result.

\begin{prop} \label{prop:rotequiv}
  For every $a_0>0$ there exist $0<a_1<b_1<a_0$ arbitrarily small and there exists $\nu\in\Z$ such that the following holds for $w\in \cA_1(a_1,b_1)\times\cA_2(c_1,d_1)$.
  \begin{itemize}
    \item[a)] If $\si>0$ then
      \[
        \Theta_1(T(c);\rho_1,\rho_2,\theta_1,\theta_2)-\theta_1 \
               \begin{cases}
                 > 2\pi\nu, & \mbox{if }\rho_1=a_1,\ (\rho_2,\theta_2)\in\cA_2^{pol}(c_1,d_1),  \\
                 < 2\pi\nu, & \mbox{if }\rho_1=b_1,\ (\rho_2,\theta_2)\in\cA_2^{pol}(c_1,d_1).
               \end{cases}
      \]
    The inequalities are reversed if $\si<0$.
    \item[b)] If $T(d)$ is strictly increasing for $d\in(c_0,d_0)$ then
      \[
        \Theta_2(T(c);\rho_1,\rho_2,\theta_1,\theta_2)-\theta_2 \
               \begin{cases}
                 > 2\pi, & \mbox{if }\rho_1\in[a_1,b_1],\ \rho_2=r_1(\theta_2),  \\
                 < 2\pi, & \mbox{if }\rho_1\in[a_1,b_1],\ \rho_2=r_2(\theta_2).
               \end{cases}
      \]
    The inequalities are reversed if $T(d)$ is strictly decreasing for $d\in(c_0,d_0)$.
  \end{itemize}
\end{prop}

\begin{proof}
We begin with the proof of part~b) because this determines the choice of $b_1$ which will then be used in the proof of part~a) where we choose $a_1$. Suppose $T(d)$ is strictly increasing for $d\in(c_0,d_0)$. For $\rho_2=r_1(\theta_2)$, that is
\[
  w_2 = \rho_2e({\theta_2}) \in \cC_{c_1} \subset \pa\cA_2(c_1,d_1),
\]
the solution $Z(t;w_2)$ of the initial value problem \eqref{eq:Z} has the period $T(c_1)$. Now Lemma~\ref{lem:existenceT} implies that $W_2(T;w)\to Z(T;w_2)$ as $w_1\to0$. Since $T(c_1)<T(c)$ the argument $\Theta_2$ of $W_2$ satisfies
\begin{equation}\label{eq:twist3}
  \Theta_2(T(c);\rho_1,\rho_2,\theta_1,\theta_2)-\theta_2 > 2\pi
\end{equation}
for $\rho_1 = |w_1|$ small. Similarly, for $\rho_2=r_2(\theta_2)$, that is
\[
  w_2 = \rho_2e({\theta_2}) \in \cC_{d_1} \subset \pa\cA_2(c_1,d_1),
\]
the solution $Z(t;w_2)$ of the initial value problem \eqref{eq:Z} has the period $T(d_1)>T(c)$, so $W_2(T(c),w)\to Z(T(c),w_2)$ as $w_1\to0$ implies
\begin{equation}\label{eq:twist4}
  \Theta_2(T(c);\rho_1,\rho_2,\theta_1,\theta_2)-\theta_2 < 2\pi
\end{equation}
for $\rho_1 = |w_1|$ small. Part~b) follows provided we choose $b_1$ so small that \eqref{eq:twist3} and \eqref{eq:twist4} hold for $\rho_1 = |w_1| < b_1$. The case that $T(d)$ is strictly decreasing for $d\in(c_0,d_0)$ can be proved analogously.

Now we can prove part~a). The proof of this part is similar to the proof of the main result in \cite{Boscaggin-Torres:2013}. Suppose first that $\si>0$. With $b_1$ determined above we choose $\nu \in \Z$ satisfying
\begin{equation}\label{eq:rot1}
  2\pi\nu > \max\left\{\Theta_1(T(c);b_1,\rho_2,\theta_1,\theta_2) - \theta_1: \theta_1\in[0,2\pi],\
                                        (\rho_2,\theta_2)\in\cA_2^{pol}(c_1,d_1)\right\}.
\end{equation}
Setting
\[
  z_1(R,\Theta) = \frac{\ka_2}{\sqrt{|\ka_2\ka_1|}} R_1 e({\Theta_1}) + R_2 e({\Theta_2}),
\]
\[
  z_2(R,\Theta) = -\frac{\ka_1}{\sqrt{|\ka_2\ka_1|}} R_1 e({\Theta_1}) + R_2 e({\Theta_2}),
\]
and
\[
\begin{aligned}
  k(R,\Theta)
   & = 2\left(\ka_2\sqrt{|\ka_1\ka_2|}\nabla_{z_1} - \ka_1\sqrt{|\ka_1\ka_2|} \nabla_{z_2}\right) g(z_1(R,\Theta),z_2(R,\Theta)) \\
   & \hspace{1cm}
         + \ka_1 \sqrt{|\ka_1\ka_2|}\nabla h(z_1(R,\Theta)) - \ka_2 \sqrt{|\ka_1\ka_2|}\nabla h(z_2(R,\Theta)),
\end{aligned}
\]
the equations for $R_1,\Theta_1$ are given by
\begin{equation}\label{eq:w_1-rot}
\left\{
  \begin{aligned}
	\dot{R}_1 &= \left\langle -J k(R,\Theta), e({\Theta_1})\right\rangle \\
	\dot{\Theta}_1 &= \frac{\ka_1\ka_2}{\pi R_1^2} + \frac{1}{R_1}\left\langle k(R,\Theta),e({\Theta_1})\right\rangle
                  =: f(R_1,R_2,\Theta_1,\Theta_2).
  \end{aligned}
  \right.
\end{equation}
Observe that
\[
  \lim_{R_1\to 0} f(R_1,R_2,\Theta_1,\Theta_2) = + \infty
\]
because
\[
  \lim_{R_1\to 0} \frac{1}{R_1}\left\langle k(R,\Theta),e({\Theta_1})\right\rangle 
   = \left\langle D^2h\left(R_2e({\Theta_2})\right)e({\Theta_1}),e({\Theta_1})\right\rangle.
\]
Thus we can choose $0<\tilde{a}_1<b_1$ such that
\begin{equation} \label{eq:theta-ineq}
  f(R,\Theta)  > \frac{2\pi\nu}{T(c)} \quad
    \text{for every } 0<R_1\le\tilde{a}_1,\ \Theta_1\in\R,\ (R_2,\Theta_2) \in \cA_2^{pol}(c_1,d_1).
\end{equation}
Then, by Lemma~\ref{lem:w1uniformly}, there exists $0<a_1<\tilde{a}_1$ such that
\[
  R_1(t;a_1,\rho_2,\theta_1,\theta_2) \leq \tilde{a}_1
   \quad\text{for every } t \in [0,T(c)],\ \theta_1\in\R,\ (\rho_2,\theta_2) \in \cA_2^{pol}(c_1,d_1).
\]
Integrating \eqref{eq:theta-ineq} on $[0,T(c)]$ gives
\begin{equation}\label{eq:rot2}
  \Theta_1(T(c);a_1,\rho_2,\theta_1,\theta_2) - \theta_1
    = \int_0^{T(c)} f\left(R(t;a_1,\rho_2,\theta_1,\theta_2),\Theta(t;a_1,\rho_2,\theta_1,\theta_2)\right) dt > 2\pi\nu
\end{equation}
for all $\theta_1\in\R$, all $(\rho_2,\theta_2)\in\cA_2^{pol}(c_1,d_1)$.
Now \eqref{eq:rot1} and \eqref{eq:rot2} imply a).

In the case $\si<0$ we choose $\nu\in\Z$ with
\[
  2\pi\nu < \min\left\{\Theta_1(T(c);b_1,\rho_2,\theta_1,\theta_2) - \theta_1: \theta_1\in[0,2\pi],\
                                        (\rho_2,\theta_2)\in\cA_2^{pol}(c_1,d_1)\right\}.
\]
The proof proceeds as above using $f(R_1,R_2,\Theta_1,\Theta_2) \to -\infty$ as $R_1\to 0$.
\end{proof}

%
%

\section{Rotation and action} \label{sec:fastrot}

The following proposition implies part c) of Theorem~\ref{thm:main}.

\begin{prop}\label{prop:shape}
  Let $z^{(n)}(t)$ be a sequence of $T$-periodic solutions of \eqref{eq:main} with the property that $z^{(n)}_1(0), z^{(n)}_2(0) \to C_0\in\Om$ and such that the solution $C(t)$ of \eqref{eq:cent-vort} with initial condition $C(0)=C_0$ is non-stationary periodic. Then setting $d_n=\big|z^{(n)}_1(0) - z^{(n)}_2(0)\big|$ the angular velocity of the difference $\displaystyle D^{(n)}(t) := z^{(n)}_1(t) - z^{(n)}_2(t) = \rho^{(n)}(t)(\cos\theta^{(n)}(t),\sin\theta^{(n)}(t))$ satisfies
  \[
    d_n^2\dot\theta^{(n)}(t) = \frac{\ka_1 \ka_2}{\pi} + o(1) \quad\text{as }n\to\infty \qquad\text{uniformly in $t$.}
  \]
\end{prop}

\begin{proof}
Define
\[
  u_n(s) := \frac{1}{d_n} D^{(n)}(d_n^2 s). 
\]
Then $u_n$ satisfies
\[
  \dot{u}_n = -\frac{\ka_1 \ka_2}{\pi}J \frac{u_n}{|u_n|^2} - o(1)\qquad\text{as $n\to\infty$, uniformly in $[0,T]$.}
\]
Note that $|u_n(0)|=1$ for all $n$, so up to a subsequence $u_n(0) \to \bar{u}$ with $|\bar{u}|=1$. By a straightforward calculation we obtain that $\frac{d}{d s} |u_n(s)|^2 = o(1)$ as $n\to\infty$, uniformly in $[0,T]$. Thus there exists $\eps>0$ such that for $n$ sufficiently large we have $|u_n(s)|\ge \eps$ uniformly for $s\in[0,T]$. Next let $u_\infty$ be the solution of the initial value problem
\[
  \left\{
  \begin{aligned}
  \dot{u}_\infty &= -\frac{\ka_1 \ka_2}{\pi} J \frac{u_\infty}{|u_\infty|^2} \\
  u_\infty(0) &= \bar{u}.
  \end{aligned}
  \right.
\]
We now deduce easily that $u_n \to u_\infty$ uniformly on $[0,T]$. Note that $\frac{d}{d s} arg(u_\infty(s)) = \frac{\ka_1 \ka_2}{\pi}$, which implies $d_n^2 \dot{\theta}^{(n)}(s) \to \frac{\ka_1 \ka_2}{\pi}$.
\end{proof}

\begin{altproof}{Theorem~\ref{thm:main}~d)}
This is a straightforward computation using
\[
  z^{(n)}_1(t) = C^{(n)}(t)+ \frac{\ka_2}{\ka_1+\ka_2} D^{(n)}(t) \quad\text{and}\quad
  z^{(n)}_2(t) = C^{(n)}(t)- \frac{\ka_1}{\ka_1+\ka_2}D^{(n)}(t),
\]
and parts a) and b) of Theorem~\ref{thm:main}.
\end{altproof}

%
%

\section{Proof of the remaining results}\label{sec:pf-rem}

\begin{altproof}{Lemma~\ref{lem:T(c)}}
First we transform the equation \eqref{eq:cent-vort} using the canonical coordinate change $(\rho,\theta)\mapsto\sqrt{2\rho} e({\theta})$. Setting $h_1(\rho,\theta) = (\ka_1+\ka_2) h(\sqrt{2\rho} e({\theta}))$ this leads to the system
\[
\left\{
\begin{aligned}
  \dot{\rho} &= -\frac{\pa}{\pa\theta} h_1(\rho,\theta) \\
  \dot{\theta} &= \frac{\pa}{\pa \rho} h_1(\rho,\theta).
\end{aligned}
\right.
\]
In convex domains the Robin function $h$ is strictly convex by \cite{Caffarelli-Friedman:1985}, hence $\frac{\pa}{\pa\rho} h_1(\rho,\theta)$ is strictly increasing in $\rho$. This implies that the minimal period $T(c)$ is decreasing with respect to $c$.

Moreover, since the origin is a nondegenerate minimum of $h$, we can apply the Hartman-Grobman Theorem, which tells us that the flow of the system near the hyperbolic critical point is topologically equivalent to the flow of the linearized system
\[
  \dot{\zeta} = -(\ka_1+\ka_2)Jh''(0)\zeta.
\]
The solution of this harmonic oscillator is periodic with period $T_m= \frac{2\pi}{|\ka_1+\ka_2|\sqrt{\det h''(0)}}$. The lemma follows.
\end{altproof}

\begin{altproof}{\ref{cor:min}}
Since $h''(z_0)$ is positive definite the Robin function is strictly convex in a neighbourhood $U$ of $z_0$. Therefore the level lines $h^{-1}(c)\cap U$ for $c>c_0=h(z_0)$ close to $c_0$ are convex. As in the proof of Lemma~\ref{lem:T(c)} the period $T(c)$ of the solution of \eqref{eq:cent-vort} with trajectory $h^{-1}(c)\cap U$ is strictly decreasing in $c$ for $c>c_0$ close to $c_0$. The corollary follows now from Theorem~\ref{thm:main}.
\end{altproof}

\begin{altproof}{\ref{cor:boundary}}
Let $\cU(\Ga)\subset\R^2$ be a tubular neighbourhood of $\Ga$ and $p:\cU(\Ga)\to\Ga$ be the orthogonal projection. Moreover let $\nu:\Ga\to\R^2$ be the exterior normal. It is well known that
\begin{equation}\label{eq:nabla-h}
  \nabla h(z) = \frac{\nu(p(z))}{2\pi d(z,\Ga)} + O(1) \quad\text{as $d(z,\Ga)=\dist(z,\Ga)\to0$;}
\end{equation}
see \cite{Bandle-Flucher:1996}. Therefore the level lines $h^{-1}(c)\cap\cU(\Ga)$ for $c>c_0$ are also strictly star-shaped with respect to $z_0$, if $c_0$ is large enough. Moreover the period $T(c)$ of the solution of \eqref{eq:cent-vort} with trajectory $h^{-1}(c)\cap\cU)\Ga)$ is strictly decreasing in $c$ due to \eqref{eq:nabla-h}. Consequently the corollary follows from Theorem~\ref{thm:main}.
\end{altproof}

\vspace{.5cm}
\noindent{\sc Address of the authors:}\\[1em]
\parbox{9cm}{Thomas Bartsch\\
 Mathematisches Institut\\
 Universit\"at Giessen\\
 Arndtstr.\ 2\\
 35392 Giessen\\
 Germany\\
 Thomas.Bartsch@math.uni-giessen.de}
\parbox{8cm}{
\vspace{5mm}
 Matteo Sacchet\\
 Dipartimento di Matematica\\
 Universit\`a di Torino \\
 via Carlo Alberto 10\\
 10123 Torino\\
 Italy\\
 matteo.sacchet@unito.it\\
 }

\end{document}